\documentclass[reqno,11pt]{amsart}
\usepackage{graphicx, enumerate, amsmath, amssymb,amsthm,manfnt,hyperref, braket}
\newtheorem{thm}{Theorem}[section]
\newtheorem{cor}[thm]{Corollary}
\newtheorem{lem}[thm]{Lemma}
\newtheorem{prop}[thm]{Proposition}
\newtheorem{result}[thm]{Result}

\theoremstyle{definition}

\theoremstyle{remark}

\numberwithin{equation}{section}
\newcommand{\norm}[1]{\left\Vert#1\right\Vert}
\newcommand{\abs}[1]{\left\vert#1\right\vert}

\newcommand{\D}{\mathbb{D}}
\newcommand{\rl}{{\mathbb{R}}}
\newcommand{\cx}{{\mathbb{C}}}
\newcommand{\B}{\mathcal{B}}
\newcommand{\dist}{{\rm dist}}
\makeatletter
\newcommand*{\shifttext}[2]{%
  \settowidth{\@tempdima}{#2}%
  \makebox[\@tempdima]{\hspace*{#1}#2}%
}
\newcommand\scvert{\scalebox{1}[.61]{$|$}}
\newcommand{\myop}{\ensuremath \raisebox{2pt}{$\scvert$}}
\newcommand{\divdif}{\ensuremath \shifttext{6.325pt}{$\Delta$}\shifttext{-6.37pt}{$\myop$}}
\makeatother
\title[H\"{o}lder estimates  and proper mappings]{H\"{o}lder estimates for Cauchy-Type Integrals and proper holomorphic mappings of symmetric products}
\author{Evan Castle}
\address{Unversity of Kentucky, Lexington, KY 40506, USA}
\email{evan.castle@uky.edu}
\author{Debraj Chakrabarti}
\address{Central Michigan University, Mt. Pleasant,  MI 48859, USA}
\email{chakr2d@cmich.edu}
\author{David Gunderman}
\address{Wabash College, Crawfordsville, IN 46077, USA}
\email{djgunder15@wabash.edu}
\author{Ellen Lehet}
\address{SUNY Potsdam, Potsdam, NY 13676, USA}
\email{lehetev195@potsdam.edu}
\date{}

\begin{document}
\maketitle
\begin{abstract}We prove estimates in H\"{o}lder spaces for some Cauchy-type integral operators representing holomorphic functions
in Cartesian and symmetric products of planar domains. As a consequence, we obtain information on the boundary regularity in H\"{o}lder spaces of proper holomorphic maps between symmetric products of planar domains.
\end{abstract}
\section{Introduction}
Let $\pi: \cx^n \to \cx^n$ be the proper holomorphic map  (called the {\em symmetrization map}) given by
\begin{equation}\label{eq-pidef}
\pi(z)=\left( \pi_1(z),\dots,\pi_n(z) \right),
\end{equation}
where $\pi_j(z)$ is the $j$-th elementary symmetric polynomial in $n$ variables. Recall that $\pi_1(z)= \sum_{k}z_k, \hspace{2mm} \pi_2(z)=\sum_{k<l}z_k z_l,$ etc. If $U$ is a bounded domain in the plane, we denote by $\Sigma^n U$, the $n$-fold {\em symmetric product} of the domain $U$ with itself, which is by definition the image of the $n$-fold Cartesian product $U^n=U\times\dots \times U$ under the map $\pi$: 
\begin{equation}\label{eq-sigmandef}
\Sigma^n U = \pi(U^n).
\end{equation}
$\Sigma^n U$ is a pseudoconvex domain in $\cx^n$ with non-Lipschitz boundary (see \cite[Proposition~5.3]{cg}). When $U=\D$, the unit disc, the 
domain $\Sigma^n \D$, under the name {\em symmetrized polydisc} arises naturally in problems of control theory (see \cite{ay1}), and 
has been studied intensively, both from the 
operator-theoretic and  function-theoretic point of view (see e.g. \cite{ay2,costara,costara2, edi1,ediz1,ay3,cg} etc).  In particular, given two bounded planar domains $U$ and $V$,
 the proper holomorphic maps from $\Sigma^n U$ to
$\Sigma^n V$ have been classified: they turn out to be functorially induced by proper holomorphic maps from $U$ to $V$ (see \cite{edi1, ediz1,cg}, and Section~\ref{sec-proper} below).

It is natural to ask how these proper holomorphic maps behave at the boundary, and this question is especially interesting since the boundaries 
of $\Sigma^n U$ and $\Sigma^n V$ are non-Lipschitz. It was shown in \cite{cg} that if $U$ and $V$ have $\mathcal{C}^\infty$-smooth boundaries, then a proper holomorphic map from $\Sigma^n U$ to $\Sigma^n V$ extends 
to a $\mathcal{C}^\infty$ map of the closures. In this paper, we consider the situation when $\partial U$ and $\partial V$ have finite order of differentiability. Let $\mathcal{C}^{k,\alpha}(U)$ denote the space of functions on a domain $U$ whose $k$-th order partial derivatives exist and  satisfy a  uniform H\"{o}lder condition of order $\alpha$. 
For a domain $U\subset\cx^n$, we denote by 
$\mathcal{A}^{k,\alpha}(U)$ the space $\mathcal{O}(U)\cap \mathcal{C}^{k,\alpha}(U)$
of holomorphic functions or maps which are of class $\mathcal{C}^{k,\alpha}(U)$.
Our main result in this direction is the following:

\begin{thm}\label{thm-proper}
Let $n\geq 1,k\geq 0$ be integers, and let 
\begin{equation}\label{eq-lambdan} \Lambda_n = \begin{cases} n! &\text{if $n\leq 3$}\\\frac{3}{2} n! &\text{if $n>3$.}\end{cases}
\end{equation}
 If  $U,V$
are planar domains with boundaries of class $\mathcal{C}^{n(k+1)+1}$ and $F:\Sigma^n U\to \Sigma^n V$ is a proper holomorphic map, then $F\in \mathcal{A}^{k,\frac{\theta}{\Lambda_n}}(\Sigma^n U)$ for each $0<\theta<1$.
\end{thm}
Note the loss of derivatives -- which is not surprising if one remembers that symmetric products have highly singular non-Lipschitz boundaries.

The proof of Theorem~\ref{thm-proper} can be reduced to the study of the regularity in H\"{o}lder spaces of certain 
integral operators closely related to the classical Cauchy operator. Let $\Gamma$ be the boundary of a smoothly bounded 
planar domain $U$,  and let $p:\cx^n\times \cx\to \cx$  be a holomorphic 
polynomial in $(n+1)$ complex variables.  We can associate with $p$ an integral operator $\mathcal{D}_p$ defined by
\begin{equation}\label{eq-dp} \mathcal{D}_p \phi(z) = \frac{1}{2\pi i} \int_\Gamma \frac{\phi(t)}{p(z,t)}dt,\end{equation}
acting on continuous functions defined on $\Gamma$. Note that $\mathcal{D}_p\phi$ is a holomorphic function on the 
open set $\cx^n\setminus \Gamma^*_p$, where 
\begin{equation}\label{eq-gammastarp}
\Gamma^*_p =\Set{w\in \cx^n \ | \ \text{ there is a $t\in\Gamma$ such that $p(w,t)=0$}}.
\end{equation}
When $n=1$ and $p(z,t)=t-z$, we obtain the classical  {\em Cauchy Transform} or the {\em Cauchy Integral}, the operator  $\mathcal{T}$ given by
\begin{equation}\label{eq-cauchytransform} \mathcal{T}\phi(z)=\frac{1}{2\pi i}\int_{\Gamma} \frac{\phi(t)}{t-z}dt,\end{equation}
which defines $\mathcal{T}\phi$ as a holomorphic function for  $z\in \cx\setminus \Gamma$. For our application, we will consider two special cases of the general integral $\mathcal{D}_p$. In our first example, we take $p= q_n$, the polynomial of degree $n$ given by
\begin{equation}\label{eq-p1}
q_n(z_1,\dots, z_n, t) = t^n-z_1t^{n-1}+z_2t^{n-2}-\dots +(-1)^n z_n.
\end{equation}
For this $q_n$, if $\cx \setminus \Gamma$ has $\kappa$ components, then one can show that $\cx^n\setminus \Gamma^*_{q_n}$ consists 
of $\binom{n+\kappa-1}{\kappa-1}$ connected components (see Section~\ref{sec-remarks}),  one of  which is precisely the symmetric product $\Sigma^nU$ defined in \eqref{eq-sigmandef}. We denote by $\mathcal{E}_n$ the operator defined by restriction of $\mathcal{D}_{q_n}$ 
\begin{equation}\label{eq-an}
\mathcal{E}_n\phi = (\mathcal{D}_{q_n} \phi)|_{\Sigma^n U},
\end{equation}
which maps continuous functions on $\Gamma$ to holomorphic functions on $\Sigma^nU$.  Note that $\mathcal{E}_1=\mathcal{T}$, the Cauchy transform, and we may call $\mathcal{E}_n$
the {\em symmetrized Cauchy transform in $n$ variables.} We prove the following regularity result for the map $\mathcal{E}_n$:
\begin{thm}\label{thm-an} Let $k\geq 0, n\geq 1$ be integers and
 let the boundary $\Gamma$ of $U$ be of class $\mathcal{C}^{nk+n+1}$. Then,  for $0<\alpha<1$,
the map $\mathcal{E}_n$ is continuous from the space $\mathcal{C}^{(k+1)n-1,\alpha}(\Gamma)$ to the space $\mathcal{A}^{k,\frac{\alpha}{\Lambda_n}}(\Sigma^n U)$,
where $\Lambda_n$ is as in \eqref{eq-lambdan}
\end{thm}
Here $\mathcal{C}^{k,\alpha}(\Gamma)$ is the space of functions on the smooth curve $\Gamma$ which are $k$
times continuously differentiable with respect to arc length, and such that the $k$-th derivative is H\"{o}lder continuous of order $\alpha$.
Our second example corresponds to the choice $p=\omega_n$, where
\begin{equation}\label{eq-p2}
\omega_n(z_1,\dots,z_n,t)= \prod_{j=1}^n(t-z_j).
\end{equation}
Assuming that $\cx\setminus \Gamma$ 
has $\kappa$ components, the space $\cx^n\setminus \Gamma_{\omega_n}^*$ has $\kappa^n$ components, 
and one of these components is the $n$-fold cartesian product $U^n$. We define the operator
$\B_n$ again by restriction to $U^n$:
\begin{equation}
\B_n \phi =  (\mathcal{D}_{\omega_n} \phi)|_{ U^n},
\end{equation}
Note that $\B_1=\mathcal{T}$, and for $n\geq 2$, we will call $\B_n$ the {\em Cauchy-N{\o}rlund Transform}  (cf. \cite[p.~199]{nor}). Denote by $\mathcal{A}^{k,\alpha}_{\rm sym}(U^n)$ the intersection $\mathcal{C}^{k,\alpha}(U^n)\cap \mathcal{O}_{\rm sym}(U^n)$, where $ \mathcal{O}_{\rm sym}(U^n)$ is the space of {\em symmetric 
holomorphic functions}.  Recall that a function $f\in \mathcal{O}(U^n)$ is 
{\em symmetric} if for each permutation $\sigma$ in the symmetric group $S_n$ and for each $z\in U^n$, we have  $ f (z)= f(\sigma(z))$, where
\begin{equation}\label{eq-sigmaaction}
 \sigma(z_1,\dots, z_n)= \left(z_{\sigma(1)},\dots, z_{\sigma(n)}\right).
\end{equation}
Our main result regarding $\B_n$  is:
\begin{thm} \label{thm-main}Let $k\geq 0, n\geq 1$ be integers and
 let the boundary $\Gamma$ of $U$ be of class $\mathcal{C}^{k+n+1}$. Then,  for $0<\alpha<1$, the  Cauchy-N{\o}rlund Transform $\B_n$ is a continuous linear map from $\mathcal{C}^{k+n-1,\alpha}(\Gamma)$ to $\mathcal{A}^{k,\alpha}_{\rm sym}(U^n)$.
\end{thm}

This paper is organized as follows. After reviewing some elementary properties of the operators $\B_n$ and $\mathcal{E}_n$ introduced above in Section~\ref{sec-remarks}, in Section~\ref{sec-bnconvex} we introduce the key observation (Proposition~\ref{prop-difference}), that $\B_n$ can be represented in terms of {\em divided differences}, discrete analogs of derivatives which occur in Newton's interpolation formula. This, along with a representation \eqref{eq-gh} of 
the divided difference allow us to prove Theorem~\ref{thm-main}  for convex domains $U$, and the general case follows in Section~\ref{sec-thmmainproof} by conformal mapping and a covering argument using the fact that bounded locally H\"{o}lder  functions are H\"{o}lder. In Section~\ref{sec-proofthman} Theorem~\ref{thm-an} is deduced from Theorem~\ref{thm-main} and a distortion estimate for the map $\pi$, which we obtain using the results of \cite{loj}. Finally, in Section~\ref{sec-proper}, a representation of proper holomorphic maps between symmetric products as an integral from \cite{cg} is used along with Theorem~\ref{thm-an} to obtain Theorem~\ref{thm-proper}. We note here that the statements  of Theorem~\ref{thm-an} and Theorem~\ref{thm-proper} can probably be improved, and slightly better regularity results may be obtained by more work. The non-optimality is because of  possible loss of information in the course of the proof. It will be interesting to obtain sharp results in this direction.

\section{Acknowledgements} This research was conducted in Summer 2014 at the NSF funded Research Experience for Undergraduates program at Central Michigan University, where Chakrabarti served as the faculty advisor and Castle, Gunderman and Lehet as  student participants in one of the research teams. We gratefully acknowledge the support of the NSF through their grant DMS 11-56890. Very special thanks are due
to Sivaram Narayan and the Mathematics Department of Central Michigan University for the excellent organization of the REU program, and to Ben Salisbury for discussions on algebraic aspects of symmetric functions.
Debraj Chakrabarti would like to thank Sagun Chanillo and Sushil Gorai for 
interesting discussions about the topic of this paper, and also the Office of Research and Sponsored Programs at Central Michigan University for their generous support through an ECI grant.

\section{Remarks on  the operators $\B_n$ and $\mathcal{E}_n$}
\label{sec-remarks}
The integral \eqref{eq-dp} defining $\mathcal{D}_p$ makes sense at $z\in \cx^n$, provided $p(z,t)$ does not vanish for any $t\in \Gamma$, i.e., 
$z\in \cx^n\setminus \Gamma^*_p$. By differentiation under the integral sign, $\mathcal{D}_p\phi$  is a holomorphic function on 
$\cx^n\setminus \Gamma^*_p$, and by writing 
\[ \cx^n \setminus \Gamma^*_p= \bigcap_{t\in\Gamma}
\Set{ z\in \cx^n | p(z,t)\not =0},\]
we see that $\cx^n\setminus \Gamma^*_p$, being the intersection of complements of complex hypersurfaces is pseudoconvex. In general, 
$\cx^n\setminus \Gamma_p^*$ is not connected, and therefore,  each of its 
components is pseudoconvex. 

Suppose that $\cx\setminus \Gamma$ consists of $\kappa$ components, $U_1,\dots, U_\kappa$, where we can assume
that $U_1=U$ (where $\partial U=\Gamma$), $U_2$ is unbounded and the other $\kappa-1$ components are bounded. 
When $p=q_n$, with $q_n$ as in \eqref{eq-p1}, we can associate with any $\kappa$-tuple of non-negative integers $(m_1,\dots, m_\kappa)$ with 
$\sum_{j=1}^\kappa m_j =n$, the subset of $\cx^n \setminus \Gamma^*_{q_n}$ given by
\[ U(m_1,\dots, m_\kappa)= \Set{z\in \cx^n| \text{ $m_j$ roots of $q_n(z,t)$ lie in $U_j$, for $j=1,\dots, \kappa$}},\]
where for a fixed $z$, we think of $q_n(z,t)$ as a polynomial in the variable $t$. An argument using continuity of roots as functions of coefficients shows that each of $U(m_1,\dots,m_\kappa)$ is connected. Since $\cx^n\setminus \Gamma^*_{q_n}$ is the disjoint union of the sets $U(m_1,\dots, m_\kappa)$ as the  $\kappa$-tuple $(m_1,\dots, m_\kappa)$  ranges over all combinations such that $\sum_{j=1}^\kappa m_j =n$, it follows that the  connected components of the set $\cx^n\setminus\Gamma^*_{q_n}$ are precisely the $U(m_1,\dots, m_\kappa)$, of which, therefore, there are a total of $  \binom{n+\kappa-1}{\kappa-1}.$
We are particularly interested in the component $U(n,0,\dots,0)$, which is  described by
\[ U(n,0,\dots, 0)= \Set{z\in \cx^n  |\text{ all roots of $q_n(z,t)$ lie in $U$}}.\]
We claim that $U(n,0,\dots,0)= \Sigma^n U =\pi(U^n)$. Indeed,   let $z\in U(n,0,\dots,0)$, and 
let $w_1,\dots,w_n$ be the roots of the polynomial $q_n(z,t)=t^n-z_1t^{n-1}+\dots(-1)^nz_n$. Then, by the definition of $U(n,0,\dots,0)$, each $w_j$ belongs to $U$. Consider the point $w \in U^n$ given by $w=(w_1,\dots,w_n).$ Then, by the relationship between the roots and coefficients of a polynomial, we see that $z=\pi(w)$ so that $z\in \Sigma^n U$. Conversely, if $z\in \Sigma^n U$, let $w=(w_1,\dots, w_n)\in U^n$ be such that $z=\pi(w)$. Then clearly $w_1,\dots, w_n\in U$ are the roots of the equation $q_n(z,t)=0$, so that $z\in U(n,0,\dots,0)$.

Note that the above argument also proves the following fact, which will be used later. If $z,w\in U^n$ are such that $\pi(z)=\pi(w)$, then there is a permutation $\sigma\in S_n$ such that $z=\sigma(w)$, where $\sigma(w)$ is as in \eqref{eq-sigmaaction}.
Analogously, note that when $p=\omega_n$, 
\begin{align*} \cx^n\setminus \Gamma^*_{\omega_n}&=\Set{z\in \cx^n|t-z_j\not =0 \text{ for all $t\in \Gamma$ and $j=1,\dots, n$}}\\
&= (\cx\setminus\Gamma)^n.
\end{align*} 
Therefore, when $\cx\setminus \Gamma$ has $\kappa$ components,  $(\cx\setminus \Gamma)^n$ has $\kappa^n$ components. In defining the operator $\mathcal{B}_n$, we restrict our attention to the component $U^n$.

 The restriction of $\mathcal{B}_n$ to $U_n$ and of $\mathcal{E}_n$ to $\Sigma^n U$ is motivated by the application 
 to proper maps of symmetric products as in Theorem~\ref{thm-proper}. It is interesting to ask what happens in the other 
 components, and whether there is a non-trivial analog of the Sokhotski-Plemelj  jump formula for these transforms. Preliminary investigations show that we should not expect boundedness in H\"{o}lder topologies of the maps $\mathcal{B}_n$ and  $\mathcal{E}_n$ when restricted to unbounded components of $\cx^n \setminus\Gamma^*_p$.  
Further, along $\partial\Sigma^n U$ and $\partial U^n$, the  $\mathcal{E}_n\phi$ and $\mathcal{B}_n\phi$ cannot 
be defined as Cauchy Principal values for H\"{o}lder continuous $\phi$ on $\Gamma$, and in fact the integral defining 
the Cauchy Principal value blows up at some points of the boundary, depending on the geometry near that point. 
If we define for $z\in \partial U^n$ the set $\Gamma(z,\rho)\subset \cx$  by $ \Gamma(z,\rho)= \bigcup_{j=1}^n B(z_j,\rho),$
where $B(z_j,\rho)$  the disc centered at $z_j$ and radius $\rho$, a computation shows that if $\phi\in \mathcal{C}^\alpha(\Gamma)$ for some $0<\alpha<1$, then 
\[  \int_{\Gamma \backslash \Gamma(z,\rho)}\frac{\phi(t)}{\omega_n(z,t)}dt = O\left(\rho^{1-\chi(z)}\right),\]
where for a point $z\in \partial U^n$, we denote by $\chi(z)$ the largest integer $k$ such that we can find  $k$  numbers  in the $n$-tuple
$(z_1,\dots, z_n)$ which are equal to each other. In particular, this shows that there is no hope for a simple generalization of the 
jump formula to the operator $\mathcal{B}_n$.

\section{A representation of $\mathcal{B}_n$ in convex domains} \label{sec-bnconvex}
We now derive some representation formulas for the Cauchy-N{\o}rlund transform, which are classical in numerical analysis (\cite{nor, stef,deboor}). Let $f$ be a holomorphic function on  $\Omega\subset\cx$.
Recall that the {\em $n$-th divided difference of $f$} is a function $f^{[n]}$ of $n+1$ variables which occurs as a coefficient in the Newton Interpolation formula (see \cite{deboor}.) For our purposes,  the divided difference  $f^{[n]}$ of order $n$ may be defined as follows: 
$f^{[0]}(z_1)= f(z_1)$, and we set
\[ f^{[1]}(z_1,z_2)= \frac{f^{[0]}(z_1)-f^{[0]}(z_2)}{z_1-z_2},\]
which is defined a priori if $z_1\not =z_2$. But the numerator is a holomorphic function on $\Omega^2$ which vanishes on the analytic variety $z_1-z_2=0$, and therefore, the right hand side extends as a holomorphic function on $\Omega^2$. We 
recursively define $f^{[n]}$ by 
\begin{equation}\label{eq-fndef} f^{[n]}(z_1,z_2,\dots, z_{n+1})= \frac{f^{[n-1]}(z_1,\dots,z_n)-f^{[n-1]}(z_2,\dots,z_{n+1}) }{z_1-z_n}.
\end{equation}
It follows as above that 
$f^{[n]}$ extends to an element of $\mathcal{O}(\Omega^{n+1})$.
We also denote by  $\divdif^{n}:\mathcal{O}(\Omega)\to \mathcal{O}(\Omega^{n+1})$
the mapping  $f\mapsto f^{[n]}$.
\begin{prop}\label{prop-difference}  The map $\mathcal{B}_n$ can be represented in terms of 
divided differences of the Cauchy transform $\mathcal{T}$ of \eqref{eq-cauchytransform} as:
\begin{equation}\label{eq-dvd}
\mathcal{B}_n = \divdif^{n-1}\circ\mathcal{T}.
\end{equation}
\end{prop}
\begin{proof} 
For $n=1$, both sides  of \eqref{eq-dvd} are equal (to the Cauchy transform.) Let $\phi$ be a continuous function on $\Gamma $ and let $f=\mathcal{T}\phi$. Assuming that for a certain $n$, we have 
$\mathcal{B}_n\phi = f^{[n-1]}$, we compute $f^{[n]}$:
\begin{align*}f^{[n]}(z_1,\dots,z_{n+1})&= \frac{1}{z_1-z_{n+1}}\left( f^{[n-1]}(z_1,\dots,z_n) 
- f^{[n-1]}(z_2,\dots,z_{n+1}) \right)\\
&=\frac{1}{2\pi i}\cdot \frac{1}{z_1-z_{n+1}}\int_{\Gamma} \phi(t)\left( \frac{1}{(t-z_1)\dots(t-z_n)}- \frac{1}{(t-z_2)\dots(t-z_{n+1})}\right)dt\\
&=\frac{1}{2\pi i}\cdot \frac{1}{z_1-z_{n+1}}\int_{\Gamma} \phi(t) \frac{(t-z_{n+1})-(t-z_1)}{(t-z_1)\dots (t-z_{n+1})}dt\\
&= \mathcal{B}_{n+1}\phi (z_1,\dots, z_{n+1}).
\end{align*}
This continues to hold when $z_1=z_{n+1}$ thanks to the fact that both sides are analytic. The result follows by induction.
\end{proof}
As a consequence we obtain a property of divided differences which is far from obvious from the representation \eqref{eq-fndef}.
\begin{cor}\label{cor-sym} If $f\in \mathcal{O}(\Omega)$, where $\Omega$ is an open subset of $\cx$, then $f^{[n-1]}\in \mathcal{O}_{\rm sym}(\Omega^{n})$, 
the space of symmetric holomorphic functions on the Cartesian power $\Omega^{n}$.
\end{cor}
\begin{proof}Given a point $(z_1,\dots,z_{n})\in \cx^{n}$, let $U\Subset \Omega$ be a smoothly bounded open set containing 
each of the points $z_1,z_2,\dots, z_{n}$ and let $\Gamma$ be the boundary of $U$.  If we set $\phi= f|_{\Gamma}$,
 by the Cauchy integral theorem, $\mathcal{B}_1\phi=f$ on $U$, where the Cauchy integral $\mathcal{B}_1\phi$ is taken along $\Gamma$. 
 But then, by Proposition~\ref{prop-difference}, $f^{[n-1]} = \mathcal{B}_n \phi$. Since $\mathcal{B}_n\phi$ is obviously in 
 $\mathcal{O}_{\rm sym}(U^{n})$, the result follows.
\end{proof}

To state the next result, we recall some definitions. For an integer $d\geq 0$, we denote by ${\mathsf{\Sigma}}_d$ the standard $d$-simplex in $\rl^{d+1}$:
\[ \mathsf{\Sigma}_d =\Set{x\in \rl^{d+1} \ | \sum_{j=1}^{d+1}x_j =1 \text{, and }\ x_j\geq 0 \text{ for $j=1,\dots,d+1$} },\]
and let  ${\mathsf{A}}_d$ be the $d$-simplex in $\rl^d$ given by
 \[ {\mathsf{A}}_d=\Set{x\in \rl^d \ | \sum_{j=1}^d x_j\leq 1 \text{, and }\ x_j\geq 0 \text{ for $j=1,\dots, d$}}.\]
Then ${\mathsf{\Sigma}}_d$ is the graph of the function 
$(x_1,\dots, x_{d})\mapsto 1- \sum_{j=1}^{d}x_j$ over ${\mathsf{A}}_d$, i.e.,
\[{\mathsf{\Sigma}}_{d} = \Set{ x\in \rl^{d+1} \ | x_{d+1}=1- \sum_{j=1}^{d}x_j \text{, and } \ (x_1,\dots, x_{d})\in {\mathsf{A}}_d}.\]
Therefore, on ${\mathsf{\Sigma}}_d$, we can take $(x_1,\dots, x_d)$ as coordinates, which we will refer to as the {\em standard coordinates} of 
${\mathsf{\Sigma}}_d$.

We denote the $d$-dimensional Hausdorff measure by $\mathcal{H}_d$. Note that since ${\mathsf{\Sigma}}_d$ is contained in an affine hyperplane
of $\rl^{d+1}$, the $d$-dimensional Hausdorff measure on ${\mathsf{\Sigma}}_d$ is simply the usual surface measure on the affine hyperplane.
Using the fact that ${\mathsf{\Sigma}}_d$ is the graph of the function $u(x_1,\dots, x_d)= 1- \sum_{j=1}^{d}x_j$, we can represent the 
Hausdorff measure on ${\mathsf{\Sigma}}_d$ in terms of the standard coordinates:
\begin{align} d\mathcal{H}_{d} &= \sqrt{1+ \sum_{j=1}^{d} \left(\frac{\partial u}{\partial x_j}\right)^2} dx_1dx_2\dots dx_{d}\nonumber\\
&= \sqrt{1+d}\cdot\,dx_1dx_2\dots dx_{d}.\label{eq-hausdorff}
\end{align}
We denote by $\langle z, w\rangle$ the standard Hermitian inner product $\sum z_j \overline{w_j}.$ We prove a complex version of a classical representation of divided differences:
\begin{prop}[Genocchi-Hermite formula] Let $U$ be a {\em convex} domain in $\cx$ and $f\in \mathcal{O}(U)$. Then,
for $z\in U^{n+1}$:\begin{equation}\label{eq-gh} f^{[n]}(z) = \frac{1}{\sqrt{n+1}} \int_{{\mathsf{\Sigma}}_{n}}f^{(n)}\left(\langle z, \tau\rangle\right)d\mathcal{H}_{n}(\tau).\end{equation}
\end{prop}
\begin{proof}We proceed by induction. When $n=0$, the left hand side is $f^{[0]}(z)=f(z)$,  $\langle z,\tau\rangle =z$,  and the right hand side is 
$ 1\cdot\int_{{\mathsf{\Sigma}}_0} f(z) d\mathcal{H}_0(\tau).$
Since in 0 dimensions the Hausdorff measure is the counting measure, the result follows. Now we assume the result for $n-1$, i.e., for $w\in U^n$, we have
\[ f^{[n-1]}(w) = \frac{1}{\sqrt{n}} \int_{{\mathsf{\Sigma}}_{n-1}}f^{(n-1)}\left(\langle w,\tau\rangle\right)d\mathcal{H}_{n-1}(\tau).\]
Then,
\begin{align}
f^{[n]}(z_1,\dots, z_{n+1})
&= \frac{1}{z_n-z_{n+1}}\left( f^{[n-1]}(z_1,\dots,z_n)- f^{[n-1]}(z_1,\dots,z_{n-1},z_{n+1})\right)\label{eq-alternative}\\
&=\frac{1}{z_n-z_{n+1}} \frac{1}{\sqrt{n}} \int_{{\mathsf{\Sigma}}_{n-1}}\left(f^{(n-1)}\left(\tau_1z_1+\dots+\tau_nz_n\right)\right.\nonumber\\
&\phantom{abcdefghijklmnokqrstuv}\left.-f^{(n-1)}(\tau_1z_1+\dots+\tau_{n-1}z_{n-1}+\tau_nz_{n+1}) \right)d\mathcal{H}_{n-1}(\tau)\nonumber\\
&= \frac{1}{\sqrt{n}} \int_{{\mathsf{\Sigma}}_{n-1}}\tau_n\left(\int_0^1f^{(n)} \left(\sum_{j=1}^{n-1}\tau_jz_j+s\tau_nz_n+(1-s)\tau_nz_{n+1} \right)ds \right)d\mathcal{H}_{n-1}(\tau)\label{eq-one}\\
&=\int_{{\mathsf{A}}_{n-1}}\tau_n\left(\int_0^1f^{(n)} \left(\sum_{j=1}^{n-1}\tau_jz_j+s\tau_nz_n+(1-s)\tau_nz_{n+1} \right)ds \right)d\tau_1d\tau_2\dots d\tau_{n-1}.\label{eq-two}
\end{align}
To obtain the representation \eqref{eq-alternative}  from the recursive definition \eqref{eq-fndef}	we first switch $z_1$ with $z_n$, followed by a use of Corollary~\ref{cor-sym} to reorder the variables in the symmetric function $f^{[n-1]}$. To obtain \eqref{eq-one} we have used the formula
\[ g(w_1)-g(w_2)=(w_1-w_2)\int_0^1 g'(sw_1+(1-s)w_2)ds,\]
which is justified since $U$ is convex and therefore the line segment joining the points $w_1, w_2\in U$ is also in $U$. In \eqref{eq-two} we used the relation \eqref{eq-hausdorff} and the fact that the point $(\tau_1,\dots,\tau_n)$ ranges over
${\mathsf{\Sigma}}_n$ as the coordinates $(\tau_1,\dots,\tau_{n-1})$ range over ${\mathsf{A}}_{n-1}$.

 Let $\theta_1,\dots, \theta_{n+1}$ denote the restrictions of the natural 
coordinates of $\rl^{n+1}$ to ${\mathsf{\Sigma}}_n$.  Therefore $(\theta_1,\dots, \theta_n)$ are standard coordinates on ${\mathsf{\Sigma}}_n$  and these coordinates range over the simplex ${\mathsf{A}}_n$  when $(\theta_1,\dots,\theta_n)\in {\mathsf{\Sigma}}_n$.
We define a map
\[ \Phi:[0,1]\times {\mathsf{\Sigma}}_{n-1}\to {\mathsf{\Sigma}}_n\]
as follows: using the coordinates $(\tau_1,\dots, \tau_{n-1})\in {\mathsf{A}}_{n-1}$ on ${\mathsf{\Sigma}}_{n-1}$ and the coordinates $(\theta_1,\dots,\theta_n)\in {\mathsf{A}}_n$ on ${\mathsf{\Sigma}}_n$, the map is given as
\[ [0,1]\times {\mathsf{A}}_{n-1}\to {\mathsf{A}}_n \text{ with } \quad (s, \tau_1,\dots, \tau_{n-1})\mapsto (\theta_1,\dots, \theta_n)  \]
where
\[ \theta_j=\tau_j \text{\quad for $j=1,\dots, n-1$}, 
\text{  and \quad}  \theta_n =  s\left(1-\sum_{j=1}^{n-1}\tau_j\right).\]
A computation shows that the absolute value of the Jacobian determinant is given by
\[ \abs{\det\Phi'(s,\tau)} = 1-{\sum_{j=1}^{n-1}\tau_j} = \tau_n.\]
Therefore, \eqref{eq-two} can be written as
\begin{align*}
f^{[n]}(z)&= \int_{[0,1]\times {\mathsf{A}}_{n-1}}f^{(n)}\left(\langle z,\Phi(s,\tau)\rangle\right)\abs{\det\Phi'(s,\tau)}ds d\tau_1\dots  d\tau_{n-1}\\
&= \int_{{\mathsf{A}}_n}f\left(\langle z,\theta\rangle\right) d\theta_1\dots d\theta_n,
\end{align*} 
where in the last line we use the change of variables formula. Finally using \eqref{eq-hausdorff} (for $d=n+1$) we conclude that
\[ f^{[n]}(z) = \frac{1}{\sqrt{n+1}} \int_{{\mathsf{\Sigma}}_{n}}f^{(n)}\left(\langle z, \theta \rangle\right)d\mathcal{H}_{n}(\theta).\]
The result now follows by induction.

\end{proof}

\section{Regularity of Divided Differences}
\subsection{Notation}
\label{sec-notation}
Recall that $\mathcal{C}^{k,\alpha}(\Gamma)$ is the space of functions on the smooth curve $\Gamma$ which are $k$
times continuously differentiable with respect to arc length, and such that the $k$-th derivative is H\"{o}lder continuous of order $\alpha$.
This linear space  
becomes a Banach space with the norm
\[ \norm{\phi}_{\mathcal{C}^{k,\alpha}(\Gamma)} = \norm{\phi}_{\mathcal{C}^k(\Gamma)} +\vert{\phi^{(k)}}\vert_{\alpha},\]
with  $ \norm{\phi}_{\mathcal{C}^k} = \sum_{j=0}^k \norm{\phi^{(j)}}_{\infty},$
where $\norm{\cdot}_{\infty}$  is the sup norm and the derivatives are taken with respect to arc length on $\Gamma$, and 
$\abs{\cdot}_\alpha$ is the H\"{o}lder semi-norm defined as
\[ \abs{\psi}_\alpha = \sup_{s\not=s'} \frac{\abs{\psi(s)-\psi(s')}}{\abs{s-s'}}.\]

In dealing with functions of several variables we use the standard multi-index conventions. Recall that $\mathcal{A}^{k,\alpha}(\Omega)$ denotes 
the space $\mathcal{C}^{k,\alpha}(\Omega)\cap \mathcal{O}(\Omega)$.  For a $\Phi\in \mathcal{A}^{k,\alpha}(\Omega)$, where $\Omega\Subset\cx^n$ is a bounded domain, we use the standard norm
\[ \norm{\Phi}_{\mathcal{C}^{k,\alpha}(\Omega)} = \norm{\Phi}_{\mathcal{C}^k(\overline{\Omega})}+ \sum_{\abs{\gamma}=k} \abs{\partial^\gamma \Phi}_\alpha,\]
where 
\[ \partial^\gamma=\left(\frac{\partial}{\partial z_1}\right)^{\gamma_1}\dots\left(\frac{\partial}{\partial z_n}\right)^{\gamma_n} \text{ and }
\norm{\Phi}_{\mathcal{C}^k(\overline{\Omega})} = \sum_{\abs{\gamma}\leq k}\norm{\partial^\gamma \Phi}_{\infty}.\]
Also, $\abs{\cdot}_\alpha$ denotes the H\"{o}lder $\alpha$-seminorm of a function given by
\[ \abs{\Psi}_\alpha = \sup_{z\not =w}{ \frac{\abs{\Psi(z)- \Psi(w)}}{\abs{z-w}^\alpha}}.\]
\subsection{ H\"{o}lder regularity of divided differences} Recall that
$\divdif^{n-1}$ denotes the map which sends a holomorphic function $f$ to its $(n-1)$-th divided difference $f^{[n-1]}$.  
\begin{lem}\label{lem-deltacont}
Let $U$ be a convex domain in $\cx$. Then, $\divdif^{n-1}$ is continuous from $\mathcal{C}^{k+n-1,\alpha}(U)$ 
to $\mathcal{C}^{k,\alpha}(U^n)$.
\end{lem}
\begin{proof} Thanks to the fact that $\mathcal{T}\left(f|_{\partial U}\right)=f$ for a function in the space
$\mathcal{A}(U)=\mathcal{O}(U)\cap \mathcal{C}(\overline{U})$ of holomorphic functions continuous up to the boundary on $U$, we have from 
\eqref{eq-dvd} that $\divdif^{n-1} f = \mathcal{B}_n f$ for $f\in \mathcal{C}^{k+n-1,\alpha}(U)$. Since $\B_n$ is given by an integral, it is clearly a closed operator, so that by the closed graph theorem, it suffices to verify that for each $f\in \mathcal{C}^{k+n-1,\alpha}(U)$,
we have $\mathcal{B}_n f\in \mathcal{C}^{k,\alpha}(U)$. Since $U$ is convex, we can use the representation \eqref{eq-gh}. For a multi-index $\gamma\in \mathbb{N}^n$, if $\abs{\gamma}\leq k$, we can repeatedly differentiate under the integral sign to obtain
\[ \partial^\gamma f^{[n-1]}(z)= \frac{1}{\sqrt{n}}\int_{{\mathsf{\Sigma}}_{n-1}}\tau^\gamma\cdot f^{(\abs{\gamma}+n-1)}\left(\langle z, \tau\rangle\right)d\mathcal{H}_{n-1}(\tau).\]
Noting that $\norm{\tau^\gamma}_\infty \leq 1$ for $\tau\in {\mathsf{\Sigma}}_{n-1}$, we have
$\norm{\partial^\gamma f^{[n-1]}}_\infty \leq C \norm{f^{(\abs{\gamma}+n-1)}}_\infty.
$
It further follows that 
if $\abs{\gamma}<k$, then each 
derivative $\partial^\gamma f^{[n-1]}$ is uniformly continuous and therefore extends to $\overline{U^n}$ continuously. Consequently,
$f^{[n-1]} \in \mathcal{C}^{k-1}(\overline{U^n}).$
Now let $z,w\in U^n$ and let $\gamma\in \mathbb{N}^n$ be a multi-index such that $\abs{\gamma}=k$. Then,
\begin{align}\abs{\partial^\gamma f^{[n-1]}(z)- \partial^\gamma f^{[n-1]}(w)}
&= \abs{ \frac{1}{\sqrt{n}}\int_{{\mathsf{\Sigma}}_{n-1}}
\tau^\gamma
\cdot
\left(f^{(\abs{\gamma}+n-1)}\left(\langle z, \tau\rangle\right)
-f^{(\abs{\gamma}+n-1)}\left(\langle w, \tau\rangle\right)\right)
d\mathcal{H}_{n-1}(\tau)}\nonumber\\
&\leq \frac{1}{\sqrt{n}}\abs{f^{(k+n-1)}}_\alpha \cdot \int_{{\mathsf{\Sigma}}_{n-1}}\tau^\gamma \abs{\langle z,\tau\rangle - \langle w, \tau\rangle}^\alpha d\mathcal{H}_{n-1}(\tau)\nonumber\\
& =  \frac{1}{\sqrt{n}}\abs{f^{(k+n-1)}}_\alpha \cdot \int_{{\mathsf{\Sigma}}_{n-1}}\tau^\gamma \abs{\langle z-w,\tau\rangle}^\alpha d\mathcal{H}_{n-1}(\tau)\nonumber\\
&\leq \frac{1}{\sqrt{n}}\abs{f^{(k+n-1)}}_\alpha \cdot \int_{{\mathsf{\Sigma}}_{n-1}}\tau^\gamma \abs{\tau}^\alpha\abs{z-w}^\alpha d\mathcal{H}_{n-1}(\tau)\nonumber\\
&= C \cdot\abs{f^{(k+n-1)}}_\alpha\cdot \abs{z-w}^\alpha.\label{eq-holder1}
\end{align}
It follows that for $\abs{\gamma}=k$, the function $\partial^\gamma f^{[n-1]}$ is H\"{o}lder continuous and therefore uniformly 
continuous on $U^n$, and therefore extends continuously to $\overline{U^n}$, so that $f^{[n-1]}\in \mathcal{C}^k(\overline{U^n})$.
Combining with the  H\"{o}lder continuity of the $k$-th partials again, we see that
$f^{[n-1]}\in \mathcal{C}^{k,\alpha}({U^n})$. The proof is complete.
\end{proof}

\maketitle
\section{Proof of Theorem~\ref{thm-main}}\label{sec-thmmainproof}
\subsection{The case $n=1$} When $n=1$, Theorem~\ref{thm-main} is a well known classical result (see \cite{hen, kress, mush}.) 

\subsection{$U$ convex, and $n\geq 2$} In this case the result follows from the representation \eqref{eq-dvd} and Lemma~\ref{lem-deltacont}. Indeed $\B_n= \divdif^{n-1}\circ \mathcal{T}$, and since we know that $\mathcal{T}:\mathcal{C}^{k+n-1,\alpha}(\Gamma)\to \mathcal{C}^{k+n-1,\alpha}(U)$ 
and $\divdif^{n-1}: \mathcal{C}^{k+n-1,\alpha}(U)\to  \mathcal{C}^{k,\alpha}(U^n)$ are continuous, the result follows.

\subsection{$U$ simply connected} Suppose that $U \subset \mathbb{C}$ is a nonconvex, simply connected domain.  Let $\phi 
\in \mathcal{C}^{k+n-1,\alpha}(\Gamma)$, where $\Gamma= \partial U$.  Let $f=\mathcal{T} \phi \in \mathcal{C}^{k+n-1,\alpha}(U)$. By the Riemann mapping theorem there exists a biholomorphism $\Psi:\D\to U$. 
The boundary regularity of $\Psi$ may be deduced from the  classical 
{\em Kellogg-Warschawski Theorem} which we will now recall. 
\begin{result}[See \cite{war}] \label{result-kw} Let $f:G\to D$ be a proper holomorphic mapping of planar domains. If for some integer $k\geq 1$ the boundaries of $G$ and $D$ are of class $\mathcal{C}^{k}$, then the mapping $f$ is of class
$\mathcal{C}^{k-1,\theta}(G)$, for each $0<\theta<1$.
\end{result}
({\em Remark:} In the classical literature, in this and other results on boundary regularity of holomorphic maps in one complex variable, it is usually assumed that $f$ is the Riemann map from a simply connected domain to the unit disc. However,  since the proofs depend only on local considerations at the boundary, these extend immediately to the more general situation of proper holomorphic mappings of domains. Recall that in one variable and for bounded domains, a proper holomorphic map is a local biholomorphism off a finite set of branch points.)

Since $\Gamma$ is of class $\mathcal{C}^{k+n+1}$, it follows  that for each $0<\theta<1$ we have $\Psi\in \mathcal{C}^{k+n,\theta}(\D)$.  Let $\Lambda =\Psi^{-1}$, then by another application of the same result, $\Lambda \in \mathcal{C}^{k+n,\theta}(U)$.  Let  $g=f \circ \Psi$, so that $f=g \circ \Lambda$.  We first wish to show that $g \in \mathcal{C}^{k+n-1,\alpha}(\mathbb{D})$.

We denote by  $\bold{\Psi}_j(z)$ the jet of order $j$ of the function $\Psi$ at the point $z$:
\[ \mathbf{\Psi}_j(z)=(\Psi(z),\Psi'(z),\dots,\Psi^{(j)}(z)).\]
Recall the Fa\`{a} di Bruno formula, which says that we may express the $\ell$-th derivative of $g$ as
\begin{equation}\label{eq-fdb} g^{(\ell)}(z)=\sum_{j=1}^\ell f^{(j)}(\Psi(z)) B_j\left( \bold{\Psi}_j(z)\right),\end{equation}
where $B_j:\cx^{\ell+1}\to \cx$ are holomorphic polynomials. (Although we do not need this information, the $B_j$'s are given explicitly as
\[ B_j(p_0, p_1,\dots,p_l)= \sum \frac{l!}{b_1!\dots b_l!}\left(\frac{p_1}{1!}\right)^{b_1}\dots \left(\frac{p_l}{l!}\right)^{b_l},\]
where the sum is over all  solutions in nonnegative integers $b_k$ to  the equations  $b_1 + 2b_2 + \dots + \ell b_\ell = \ell$ and  $b_1 + \dots + b_{\ell} =j.$) If we introduce the function  $F_j$ given by
$F_j(p_0,\dots,p_\ell) = f^{(j)}(p_0) B_j(p_1,\dots,p_\ell),$
we can rewrite \eqref{eq-fdb} as
\[ g^{(\ell)}(z)= \sum_{j=1}^{\ell} F_{j}(\bold{\Psi}_\ell(z)).\]
Note that $F_{j} \in \mathcal{C}^{\alpha}(U^\ell)$  if $j\leq k+n-1$, because $F_j$ is the product of H\"{o}lder continuous functions of order $\alpha$. Then, we have,
\begin{align*}
\abs{g^{(k+n-1)}(z)-g^{(k+n-1)}(w)} &=\abs{\sum_{j=1}^{k+n-1} F_{j}(\mathbf{\Psi}_{k+n-1}(z))-F_{j}(\mathbf{\Psi}_{k+n-1}(w))}\\
&\leq \sum_{j=1}^{k+n-1} \abs{F_{j}(\mathbf{\Psi}_{k+n-1}(z))-F_{j}(\mathbf{\Psi}_{k+n-1}(w))}\\
&\leq  \sum_{j=1}^{k+n-1}\abs{F_j}_\alpha \abs{\mathbf{\Psi}_{k+n-1}(z) - \mathbf{\Psi}_{k+n-1}(w)}^{\alpha} &\text{(using the H\"{o}lder condition)}\\
& = C \abs{\mathbf{\Psi}_{k+n-1}(z) - \mathbf{\Psi}_{k+n-1}(w)}^{\alpha}.\\
&\leq \ C \abs{z-w}^\alpha,
\end{align*}
where the last step holds since $\bold{\Psi}_{k+n-1}\in \mathcal{C}^{1,\theta}(\D)$ for each $0<\theta<1$.
It follows therefore  that $g \in \mathcal{C}^{k+n-1,\alpha}(\mathbb{D})$.

Recall that $f^{[n-1]}=(g\circ \Lambda)^{[n-1]}$ where $\Lambda = \Psi^{-1}$ is the Riemann map from $U$ to $\D$,
and $\Lambda$ is in $\mathcal{C}^{k+n,\theta}(U)$ for $0<\theta<1$. 
  We use an analog of the Fa\`{a} di Bruno formula for divided differences. By  \cite[Theorem~1]{flo}, at a point $z\in U^{n}$ the $(n-1)$-th divided difference of the  composite function $f=g\circ\Lambda$  is given by
\begin{equation}\label{eq-dfdb}
f^{[n-1]}(z)=\sum_{j=1}^{n-1}\Lambda^{\{j,n-1\}}(z)\cdot g^{[j]}(\bold{\Lambda}_j(z)),
\end{equation}
where we have $\bold{\Lambda}_j(z)=(\Lambda(z_1), \Lambda(z_2), \dots, \Lambda(z_{j+1}))$
(a discrete analog of the $j$-jet), and 
the $\Lambda^{\{j,n-1\}}$ are the functions on $U^{n}$ given by
\[ \Lambda^{\{j,n-1\}}(z)= \sum_{j=1}^{n-1}\left(\Lambda^{[1]}(z_1,\cdot)\Lambda^{[1]}(z_2,\cdot)\dots\Lambda^{[1]}(z_j,\cdot) \right)^{[n-j-1]}(z_{j+1},\dots,z_n). \]
Note that  $\Lambda\in \mathcal{A}^{k+n, \theta}({U})$, we claim that:
\begin{equation}\label{eq-lajn1}
\Lambda^{\{j,n-1\}}\in\mathcal{A}^{k+j,\theta}({U^{n}}).
\end{equation}
Since $\Lambda^{\{j,n-1\}}$ is a polynomial in divided differences of $\Lambda$ of order at most $(n-j)$, the claim would
follow if we can show that $\Lambda^{[\ell]}\in \mathcal{A}^{k+n-\ell,\theta}(U^{\ell+1})$. This would follow from 
Lemma~\ref{lem-deltacont} if $U$ were convex, but we will instead use the fact that $\Lambda$ is a biholomorphism,
and its inverse $\Psi$ is indeed defined on the convex domain $\D$, and therefore Lemma~\ref{lem-deltacont} does apply 
to it. By repeatedly taking divided differences, we see that 
\[ \Lambda^{[\ell]}(z_1,\dots, z_{\ell+1}) = \frac{P}{\prod_{i<j}\Psi^{[1]}(\Lambda(z_i), \Lambda(z_j))},\]
where $P$ is a polynomial in divided differences of $\Psi$ of order at most $\ell$, evaluated at the points $\Lambda(z_1), \dots, \Lambda(z_{\ell+1})$. Note that by injectivity of $\Psi$, the 
denominator is a nonvanishing function on $U^{\ell+1}$ of class $\mathcal{A}^{k+n-\ell, \theta}(U^{\ell+1})$. Claim \eqref{eq-lajn1} follows, since, products of H\"{o}lder functions and reciprocals of nonvanishing H\"{o}lder functions are H\"{o}lder of 
the same class.

 Since $g\in \mathcal{O}(\mathbb{D})$, and $\D$ is convex, we can represent the function $g^{[j]}$ using the  Genocchi-Hermite formula. Putting this into \eqref{eq-dfdb}, we obtain
\[
f^{[n-1]}(z)=\sum_{j=1}^{n-1}\frac{\Lambda^{\{j,n-1\}}(z)}{\sqrt{j+1}}\int_{{\mathsf{\Sigma}}_j}g^{(j)}(\langle  \bold{\Lambda}_j(z),\tau \rangle)d\mathcal{H}_j(\tau)\]
Let $\gamma$ be a multi-index with $\abs{\gamma}=k$. Differentiating under the integral sign and using the Leibniz rule for partial derivatives,
\begin{align*}
\partial^{\gamma} f^{[n-1]}(z)
&=\sum_{j=1}^{n-1}\frac{1}{\sqrt{j+1}}\sum_{0\leq\beta\leq\gamma}\binom{\gamma}{\beta}\partial^{\gamma-\beta}\Lambda^{\{j,n-1\}}(z)\cdot\partial^{\beta}\left(\int_{{\mathsf{\Sigma}}_j}g^{(j)}(\langle  \bold{\Lambda}_j(z) ,\tau\rangle)d\mathcal{H}_j(\tau)\right)\\
&=\sum_{j=1}^{n-1}\frac{1}{\sqrt{j+1}}\sum_{0\leq\beta\leq\gamma}\binom{\gamma}{\beta}\partial^{\gamma-\beta}\Lambda^{\{j,n-1\}}(z)\left(\int_{{\mathsf{\Sigma}}_j}g^{(j+\abs{\beta})}(\langle  \bold{\Lambda}_j(z),\tau \rangle)\cdot\tau^\beta\cdot ({\bold{\Lambda}_{n-1}'}(z))^\beta d\mathcal{H}_j(\tau)\right),
\end{align*}
where ${\bold{\Lambda}_{n-1}'}(z)=({{\Lambda}'}(z_1),\dots,{{\Lambda}'}(z_n))\in\cx^n$. If we define for $0\leq\ell\leq k$ and 
$z\in U^n$,
\[ \mathsf{V}_\ell(z)= \sum_{\substack{0\leq\beta\leq\gamma\\ \abs{\beta}=\ell}}\binom{\gamma}{\beta}\partial^{\gamma-\beta}\Lambda^{\{j,n-1\}}(z)\cdot ({\bold{\Lambda}_{n-1}'}(z))^\beta, \]
then by \eqref{eq-lajn1},  $\partial^{\gamma-\beta}\Lambda^{\{j,n-1\}}\in \mathcal{A}^{\abs{\beta},\theta}(U^n)$, and since $\Lambda$ is in $\mathcal{C}^{k+n,\theta}(U)$ , we conclude that $ {\bold{\Lambda}_{n-1}'}\in \mathcal{A}^{k+n-1,\theta}(U^n)$. It follows that for each $0\leq \ell \leq k$ and for each $0<\theta<1$, we have
$\mathsf{V}_\ell \in \mathcal{A}^\theta(U^n).$
In terms of the $\mathsf{V}_\ell$, we can write
\begin{equation}\label{eq-fgamma}\partial^{\gamma} f^{[n-1]}(z)= \sum_{j=1}^{n-1}\sum_{\ell=1}^k\frac{1}{\sqrt{j+1}}\mathsf{V}_\ell(z)\cdot\int_{{\mathsf{\Sigma}}_j}g^{(j+\ell)}(\langle  \bold{\Lambda}_j(z),\tau \rangle)h_\ell^{j+1}(\tau) d\mathcal{H}_j(\tau),
\end{equation}
where $h^{j+1}_\ell(\tau) = \sum_{\abs{\beta}=\ell} \tau^\beta$ is the complete symmetric polynomial in $(j+1)$ variables of 
degree $\ell$. Now
\begin{align*}
&\phantom{=}\abs{\int_{{\mathsf{\Sigma}}_j}g^{(j+\ell)}(\langle  \bold{\Lambda}_j(z),\tau \rangle)h_\ell^{j+1}(\tau) d\mathcal{H}_j(\tau)- \int_{{\mathsf{\Sigma}}_j}g^{(j+\ell)}(\langle \bold{\Lambda}_j(w) ,\tau\rangle)h_\ell^{j+1}(\tau) d\mathcal{H}_j(\tau)}\\
& = 
\abs{\int_{{\mathsf{\Sigma}}_j}\left(g^{(j+\ell)}(\langle \bold{\Lambda}_j(z),\tau \rangle)-g^{(j+\ell)}(\langle  \bold{\Lambda}_j(w), \tau\rangle) \right)h_\ell^{j+1}(\tau) d\mathcal{H}_j(\tau)}\\
& \leq C\abs{\int_{{\mathsf{\Sigma}}_j}\abs{\langle \bold{\Lambda}_j(z),\tau \rangle - \langle \bold{\Lambda}_j(w),\tau \rangle}^\alpha h_\ell^{j+1}(\tau) d\mathcal{H}_j(\tau)}\\
& \leq C\abs{\int_{{\mathsf{\Sigma}}_j}\abs{\langle \bold{\Lambda}_j(z) -\bold{\Lambda}_j(w),\tau\rangle }^\alpha h_\ell^{j+1}(\tau) d\mathcal{H}_j(\tau)}\\
& \leq C\abs{\int_{{\mathsf{\Sigma}}_j}\abs{\tau}^\alpha\abs{ \bold{\Lambda}_j(z) -\bold{\Lambda}_j(w) }^\alpha h_\ell^{j+1}(\tau) d\mathcal{H}_j(\tau)}\\
&\leq C\abs{ \bold{\Lambda}_j(z) -\bold{\Lambda}_j(w)}^\alpha\\
&\leq C \abs{z-w}^\alpha.
\end{align*}
Therefore the integral factor in each term of the double sum in \eqref{eq-fgamma}
is a H\"{o}lder continuous function of order $\alpha$, and since $\mathsf{V}_\ell$ is also H\"{o}lder continuous of order $\alpha$, we see that $\partial^\gamma f^{[n-1]}\in \mathcal{A}^\alpha(U^n)$. 
It now follows that $f^{[n-1]}=\mathcal{B}_n\phi\in\mathcal{C}^{k,\alpha}(U^n)$.

\subsection{Multiply Connected Case}
Now, let $U$ be a domain with $h$ holes, that is, $U$ is $(h+1)$-connected.
We claim that {\em we can find $h+1$ simply connected domains $\{U_j\}_{j=1}^{h+1}$, each with boundary of class $\mathcal{C}^{n+k+1}$ and an $R>0$, such that
\begin{enumerate}
\item $\displaystyle{\bigcup_{j=1}^{h+1} U_j = U}$, and 
\item if $z,w\in U$ are such that $\abs{z-w} < R$,  for at least one $j$ we have $z,w \in U_j$.
\end{enumerate}}

Assuming the claim,  let as before $f=\mathcal{B}_1\phi$. By the previous argument for the simply-connected case, we see that $\left(f|_{U_j}\right)^{[n-1]} \in \mathcal{C}^{k,\alpha}(U_j)$. Therefore, in particular, $f^{[n-1]}$ is bounded on $\overline{U^n}$. Further, if $\abs{z-w}<R$, we see that for any multi-index $\gamma$ with $\abs{\gamma}=k$, we have that $\abs{\partial^{\gamma}f^{[n-1]}(z) - \partial^{\gamma}f^{[n-1]}(w)} \leq C\abs{z-w}^{\alpha}$, since such $z,w$ would belong to a common $U_j$ by the claim. Therefore, applying \cite[Lemma~7.3]{kress}, we conclude that $f^{[n-1]} = \mathcal{B}_n\phi \in \mathcal{C}^{k,\alpha}(U^n)$. This completes the proof of Theorem~\ref{thm-main},  provided we justify the claim made above.
\subsection{Proof of Claim}
Given two subsets $V,W$ of the plane, we say that they are {\em well-separated} if we have
\[ \overline{V\setminus W}\cap \overline{W\setminus V} =\emptyset.\] 
It is not difficult to see that Condition (2) of the claim is equivalent to demanding that each pair of simply-connected sets
 $U_j$ and $U_\ell$ in the collection $\{U_j\}$ is well-separated. Note further that the condition of well-separatedness is 
 invariant under homeomorphisms, which we will take advantage of by replacing the domain $U$ with homeomorphic
 domains with simpler geometry. Indeed if $\Psi:\overline{U}\to \overline{G}$ is a diffeomorphism of class $\mathcal{C}^{k+n+1}$ then it suffices to find a cover of $G$ by simply connected open sets $\{G_i\}_{i=1}^{h+1}$,
 each with $\mathcal{C}^{k+n+1}$ boundary, and such that each pair of sets $G_i$ and $G_j$ is well-separated.  We will 
 illustrate the fact that this is always possible by using pictures.
 We may take $G$ to look like in the picture below (in this case we take $h=2$):

\begin{center}
\includegraphics[scale=.5]{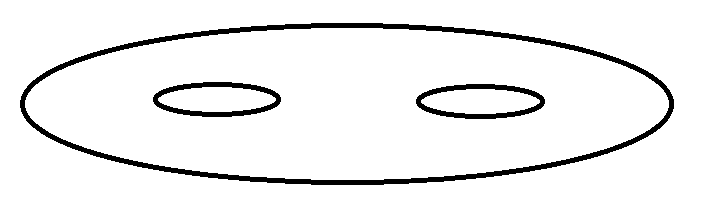}
\end{center}

 And it is clear how to construct the sets $G_1, G_2, G_3$ in such a way that they have $\mathcal{C}^{n+k+1}$ boundary and are pairwise well-separated. In the picture below, we have drawn the boundaries of $G_1$ and $G_3$ in dotted lines and that of $G_2$ in bold lines to illustrate the idea:
 \begin{center}
 \includegraphics[scale=.5]{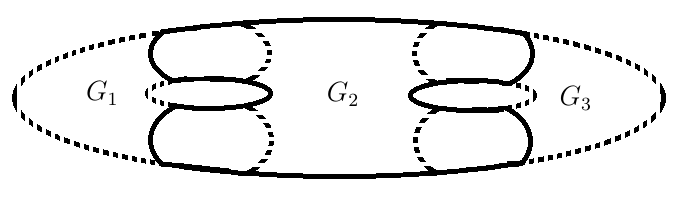}
\end{center}
The following picture shows the three domains $G_1,G_2,G_3$ separately:
\begin{center}
\includegraphics[scale=.5]{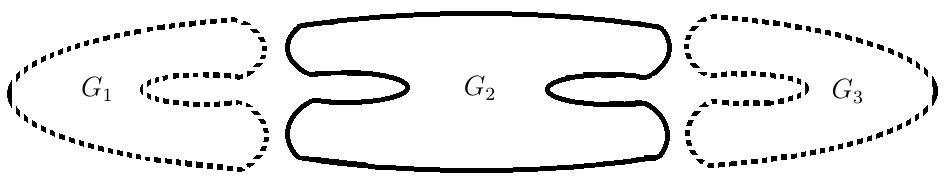}
\end{center}

\section{Proof of Theorem \ref{thm-an}}\label{sec-proofthman}
For $z,w\in \cx^n$, define
\[\delta(z,w)= \min_{\sigma \in S_n}\abs{z-\sigma(w)},\]
where $\sigma(w)$ is as in \eqref{eq-sigmaaction}. This may be thought of as a natural metric on 
the quotient $\cx^n/{S_n}$.
We compare the distance induced in this way on the symmetric product $U^n/{S_n}$ with the metric on the realization $\pi(U^n)=\Sigma^n U$.
We will use the global {\L}ojasiewicz inequality of Ji-Koll\'{a}r-Shiffman (see \cite[Corollary 6, Equation 6.2]{loj}), which states the following.
Let $f_1,\dots,f_n\in \mathbb{C}[z_1,\dots,z_N]$ and for $1\leq i \leq n$,  let $ d_i= {\rm deg}f_i$. Let $Z\subset \mathbb{C}^N$ be the common zero set of the polynomials $\{f_i\}_{i=1}^n$. Then there is a constant $C>0$ such that
\begin{equation}\label{eq-gloj}
\left(\frac{{\rm dist}(z,Z)}{1+\abs{z}^2}\right)^{\overline{B}(N,d_1,\dots,d_n)}\leq C \cdot \max_i\abs{f_i(z)},
\end{equation}
where the constant $\overline{B}(N,d_1,\dots,d_n)$ can be computed explicitly (see \cite[Section~3]{loj}.)

\begin{lem}\label{lem-loja} For any domain $U \subset \cx$, there is a $C > 0$, such that for $z,w \in U^n$, we have \[\delta(z,w)^{\Lambda_n}\leq C\abs{\pi(z)-\pi(w)},\]
where $\Lambda_n$ is as in \eqref{eq-lambdan}.
\end{lem}
\begin{proof}
We will use \eqref{eq-gloj} when $N=2n$ with coordinates $(z,w)$ on $\cx^{2n}=\cx^n \times \cx^n$, where $z,w \in \cx^n.$  For $1\leq i \leq n$, we set $ f_i(z,w) = \pi_i(z)-\pi_i(w),$
so that $d_i=i.$ In the notation of Section 3 of \cite{loj}, we have
\[\overline{B}(2n,d_1,\dots,d_n)=\left(\frac{3}{2}\right)^j B(2n,d_1,\dots,d_n)+\theta \]
Since $d_i=i$ for $i\leq n$, we have $B(2n, d_1,\dots, d_n)= n!$.
We also have that $\theta =0$ since the dimension of the space is $2n$ and the number of equations is $k=n$.  Then $j= \#\Set{i< n-1|i=2}$, so that $j=0$ if $n \leq 3$ and $j=1$ otherwise.
Thus 
\[\overline{B}(2n,d_1,\dots,d_n) = \Lambda_n = \begin{cases} n! &\text{if $n\leq 3$}\\\frac{3}{2} n! &\text{if $n>3$.}\end{cases}\]

Therefore, we conclude that 
\begin{equation}\label{eq-loja}
\left(\frac{\dist\left((z,w),Z \right)}{1+\abs{z}^2+\abs{w}^2} \right)^{\Lambda_n}\leq C \cdot \max_i\abs{\pi_i(z)-\pi_i(w)},
\end{equation}
where $Z=\Set{ (z,w)\in \cx^n \times \cx^n \ | \ \pi_i(z)- \pi_i(w) =0 \text{ for } 1 \leq i \leq n}.$ Clearly if $(z,w)\in Z$ then for some $\sigma\in S_n$ we have $z=\sigma(w)$. Let us denote for $\sigma\in S_n$
\[L_\sigma = \Set{(z,w)\in\cx^n\times \cx^n \ | \ z = \sigma(w)},\]
which is an $n$-dimensional linear subspace of the complex vector space $\cx^{2n}$. We can alternatively represent $L_\sigma$ as:
\begin{align}
L_{\sigma} &= \Set{(z,w)\in\cx^n\times \cx^n \ | \ z_j - w_{\sigma(j)} = 0 \text{ for } 1 \leq j \leq n}\nonumber\\
&=\Set{(z,w)\in\cx^n\times \cx^n \ | \ \langle(z,w), e_j-e_{\sigma(j)+n}\rangle = 0, 1\leq j \leq n}\label{eq-lsigma2}
\end{align}
where $\{e_j\}_{j=1}^{2n}$ denotes the standard complex basis of $\cx^{2n}.$ We may write the affine algebraic variety $Z$ as the union of irreducible algebraic subvarieties as follows:
\begin{equation}
Z=\bigcup_{\sigma\in S_n}L_\sigma\label{eq-zlsigma}
\end{equation}
Then it is clear by \eqref{eq-lsigma2} that the orthogonal complement to $L_{\sigma}$ (in $\cx^n\times\cx^n$) is given by
\[L_{\sigma}^\perp = {\rm span} \left\{\frac{1}{\sqrt{2}}(e_j-e_{\sigma(j)+n}), 1\leq j \leq n\right\}.\]

\noindent The length of the projection of a point $(z,\sigma(w))$ onto $L_{\sigma}^\perp$ is given by:
\begin{align*}
\abs{P_{L_{\sigma}^{\perp}}(z,w)} 
&= \abs{\sum_{j=1}^n \left\langle (z,w), \frac{1}{\sqrt{2}}(e_j - e_{\sigma(j)+n})\right\rangle  \frac{1}{\sqrt{2}}(e_j - e_{\sigma(j)+n})} \\
&= \frac{1}{2}\abs{\sum_{j=1}^n (z_j - w_{\sigma(j)})(e_j - e_{\sigma(j)+n})} \\
&= \frac{1}{2} \left( \sum_{j=1}^n 2\abs{z_j - w_{\sigma(j)}}^2 \right)^{\frac{1}{2}} \\
&= \frac{1}{\sqrt{2}} \abs{z - \sigma(w)}.
\end{align*}
Using \eqref{eq-zlsigma}:
\begin{align*}
{\rm dist}((z,w), Z) &= \min_{\sigma\in S_n} {\rm dist}((z,w), L_\sigma)\\
&= \min_{\sigma\in S_n} \abs{P_{L_{\sigma}^{\perp}}(z,w)}\\
&= \min_{\sigma\in S_n} \frac{1}{\sqrt{2}}\abs{z-\sigma(w)}\\
&=\frac{1}{\sqrt{2}}\delta(z,w).
\end{align*}
Substituting the last equality  in \ref{eq-loja} we obtain:
\[\left(\frac{\frac{1}{\sqrt{2}} \delta(z,w)}{1+\abs{z}^2+\abs{w}^2} \right)^{\Lambda_n}\leq C \cdot \max_i\abs{\pi_i(z)-\pi_i(w)}\]

Since the point $(z,w)$ belongs to the bounded set $U^n\times U^n\subset\cx^{2n}$ we have
\begin{align*}
\delta(z,w)^{\Lambda_n}&\leq C \cdot \max_i\abs{\pi_i(z)-\pi_i(w)}\\
&\leq C \cdot \abs{\pi(z)-\pi(w)},
\end{align*}
which completes the proof.
\end{proof}

\subsection{Three operators} We introduce three operators $\pi_\ast, M_\gamma$ and $j_k^\ast$
 which will be needed in the proof of Theorem~\ref{thm-an}.  Let $\mathcal{O}_{\rm sym}(U^n)$ be the space of symmetric holomorphic functions on $U^n$. Then we define a push-forward map $\pi_{\ast}:\mathcal{O}_{\rm sym}(U^n)\to\mathcal{O}(\Sigma^nU)$ in the following way:
for an $f\in \mathcal{O}_{\rm sym}(U^n)$, \[(\pi_{\ast}f)(z)=f(\zeta),\] where $\zeta \in U^n$ is any point such that $\pi(\zeta)=z$.
\begin{prop}\label{pi-ast}
 $\pi_{\ast}$ is well-defined and maps $\mathcal{O}_{\rm sym}(U^n)$ to $\mathcal{O}(\Sigma^nU)$.
\end{prop} 
\begin{proof}
If $\zeta$ and $\zeta'$ are such that $\pi(\zeta)=\pi(\zeta')=z$ then $\zeta'=\sigma(\zeta)$ for some $\sigma\in S_n$ (see Section~\ref{sec-remarks}). Therefore, $f(\zeta)=f(\zeta')$ as $f$ is symmetric.  This shows that the definition makes sense and $\pi_\ast f$ is a well defined function on $\Sigma^nU$.

Now we show that $\pi_\ast f$ is holomorphic. Off the analytic set 
\[Z= \Set{ z\in U^n \ | \ \displaystyle{\prod_{i<j} (z_i-z_j) = 0}}\]
the map $\pi$ is a local biholomorphism. It follows that $\pi_\ast f$ is holomorphic on $\pi(U^n\setminus Z) \supset \Sigma^nU\setminus \pi(Z)$, where $\pi(Z)$ is an analytic set by the Remmert Proper Mapping Theorem. Since $\pi_\ast f$ is bounded, by the Riemann Removable Singularity Theorem it extends holomorphically to $\Sigma^nU$.
\end{proof}
For a multi-index $\gamma\in \mathbb{N}^n$, we define $M_\gamma$ to be the operator which multiplies a function on $\Gamma$ by the smooth function
\begin{equation}\label{eq-ugamma}
u_{\gamma}(t)= (-1)^{\abs{\gamma}} \cdot \abs{\gamma}!\cdot t^{\sum_{j=1}^n \gamma_j(n-j)},
\end{equation}
so that
\[ (M_\gamma \phi)(t)=u_\gamma(t)\phi(t). \]
For an integer $k\geq 0$, let $j_k$ denote the diagonal embedding of $\cx^n$ in the $k$-fold product $\cx^{nk}= (\cx^n)^k=\cx^n\times\dots\times\cx^n$, i.e., for a $z\in \cx^n$, we have
\[ j_k(z) = (z,\dots,z)\in \underbrace{\cx^n\times \dots\times\cx^n}_{\text{$k$ times}}.\]
Let $\Omega\subset \cx^n$, and let $\Omega^k = \Omega\times\dots\times\Omega\subset \cx^{nk}$. For a function $f$ on $\Omega$, we define the function $j^*_k f$ on $\Omega^k$ by
\begin{equation}\label{eq-jstar} j^*_k f=f\circ j_k.\end{equation}
  
\subsection{End of Proof of Theorem~\ref{thm-an}} A direct computation shows that
\[ \partial^\gamma \left( \frac{1}{q_n(z,t)}\right)= \frac{u_\gamma(t)}{q_n(z,t)^{\abs{\gamma}+1}},\]
Let $z\in \Sigma^n U$ and let  $w\in U^n$ be such that $z=\pi(w)$.  For a continuous $\phi$ on $\Gamma$, we have  :
\begin{align*}
\partial^\gamma\mathcal{E}_n\phi(z)  &=\frac{1}{2\pi i} \int_\Gamma \frac{M_\gamma \phi(t)}{q_n(z,t)^{\abs{\gamma}+1}}dt\\
&= \frac{1}{2\pi i} \int_\Gamma \frac{M_\gamma \phi(t)}{q_n(\pi(w),t)^{\abs{\gamma}+1}}dt\\
&= \frac{1}{2\pi i} \int_\Gamma \frac{M_\gamma \phi(t)}{\left((t-w_1)\dots (t-w_n)\right)^{\abs{\gamma}+1}}dt\\
&=  j^*_{\abs{\gamma}+1} \left( \mathcal{B}_{n(\abs{\gamma}+1)} \left( M_{\gamma}\phi\right)\right)(w).
\end{align*}
Since $z=\pi(w)$  it follows that 
\begin{equation}\label{eq-dgammaen}\partial^\gamma \circ\mathcal{E}_n = \pi_* \circ j^*_{\abs{\gamma}+1} \circ \mathcal{B}_{n(\abs{\gamma}+1)} \circ M_{\gamma}.\end{equation}
To prove Theorem~\ref{thm-an}, it suffices to show that for each multi-index $\gamma$ with  $\abs{\gamma}=k$, the operator $\partial^\gamma\mathcal{E}_n$ maps the space $\mathcal{C}^{n(k+1)-1,\alpha}(\Gamma)$ 
continuously into $\mathcal{C}^{\frac{\alpha}{\Lambda_n}}(\Sigma^n)$.
By \eqref{eq-dgammaen}, we have
$\partial^\gamma \circ\mathcal{E}_n = \pi_* \circ j^*_{k+1} \circ \mathcal{B}_{n(k+1)} \circ M_{\gamma}.$
The operator $M_\gamma$, being multiplication by the smooth function $u_\gamma$, maps
$\mathcal{C}^{n(k+1)-1,\alpha}(\Gamma)$ continuously to itself.  Since the boundary $\Gamma$ is of class $\mathcal{C}^{n(k+1)+1}$, it follows by Theorem~\ref{thm-main}, that  the operator
$\mathcal{B}_{n(k+1)}$ maps $\mathcal{C}^{n(k+1)-1,\alpha}(\Gamma)$  continuously to $\mathcal{C}^{0,\alpha}(U^{n(k+1)})$. Note that $j^*_k$, defined by \eqref{eq-jstar} maps  $\mathcal{C}^{0,\alpha}(U^{n(k+1)})$ to $\mathcal{C}^{0, \alpha}(U^n)$. Indeed,
  \begin{align*}\abs{j_k^* f(z)- j_k^* f(w)} & = \abs{ f(j_k(z)) -f( j_k(w)) }\\
  &\leq  \abs{ f}_\alpha  \abs{j_k(z)-j_k(w)}^\alpha\\
  &= \abs{f}_\alpha  \abs{j_k(z-w)}^\alpha\\
  & =\abs{f}_\alpha n^{\frac{\alpha}{2}} \abs{z-w}^\alpha.
  \end{align*} Therefore, $j_k^*$ maps $\mathcal{A}^\alpha(U^{n(k+1)})$ continuously to $\mathcal{A}^\alpha(U^n)$.
To complete the proof it suffices to show that $\pi_\ast$ maps $\mathcal{A}^\alpha(U^n)$ continuously to $\mathcal{A}^{\frac{\alpha}{\Lambda_n}}(\Sigma^nU)$. By Proposition~\ref{pi-ast}, $\pi_\ast$ preserves 
holomorphicity, so it suffices to show that $\pi_{\ast}$ is continuous from $\mathcal{C}^{\alpha}(U^n)$ to $\mathcal{C}^{\frac{\alpha}{\Lambda_n}}(\Sigma^nU)$. If $\pi(\zeta) = z$ and $\pi(\lambda) = w$ we have
\begin{align*}
\abs{\pi_{\ast}f(z)- \pi_{\ast}f(w)}&=\abs{f(\zeta)-f(\lambda)}\\
&\leq\abs{f}_{\alpha}\abs{\zeta-\lambda}^{\alpha}
\end{align*}
Then, for any $\sigma \in S_n$, we know that $w=\pi(\sigma(\lambda))$. Therefore, the relation
\[\abs{\pi_{\ast}f(z)- \pi_{\ast}f(w)}\leq\abs{f}_{\alpha}\abs{\zeta-\sigma(\lambda)}^{\alpha}\] holds for every $\sigma\in S_n$. Hence
\begin{align*}
\abs{\pi_{\ast}f(z)- \pi_{\ast}f(w)}&\leq \abs{f}_{\alpha}\min_{\sigma\in S_n}\abs{\zeta-\sigma(\lambda)}^{\alpha}\\
&=\abs{f}_{\alpha}\delta(\zeta,\lambda)^{\alpha}\\
&\leq C\cdot \abs{f}_{\alpha} \abs{\pi{(\zeta)}-\pi{(\lambda)}}^{\frac{\alpha}{\Lambda_n}}&\text{By Lemma~\ref{lem-loja}}\\
&=C\cdot \abs{f}_{\alpha}\abs{z-w}^{\frac{\alpha}{\Lambda_n}}.
\end{align*}

\noindent Thus
\begin{align}
\abs{\pi_\ast f}_\frac{\alpha}{\Lambda_n}&= \sup_{z\neq w} \frac{\abs{\pi_\ast f(z)-\pi_\ast f(w)}}{\abs{z-w}^{\frac{\alpha}{\Lambda_n}}}\nonumber\\
&\leq C\abs{f}_\alpha\label{eq-seminormf}
\end{align}
Also we have
\begin{align}
\norm{\pi_\ast f}_\infty &= \sup_{z\in\Sigma^nU}\abs{\pi_\ast f(z)}\nonumber\\
&=\sup_{z\in\Sigma^nU}\abs{f(\pi^{-1}(z))}\nonumber\\
&=\sup_{\zeta\in U^n}\abs{f(\zeta)}\nonumber\\
&=\norm{f}_\infty\label{norm-f}
\end{align}

\noindent Combining equations \eqref{eq-seminormf} and \eqref{norm-f} we obtain
\begin{align*}
\norm{\pi_\ast f}_{\mathcal{C}^{\frac{\alpha}{\Lambda_n}}(\Sigma^nU)}&=\abs{\pi_\ast f}_{\frac{\alpha}{\Lambda_n}}+\norm{\pi_\ast f}_\infty\\
&\leq C\abs{f}_\alpha + \norm{f}_\infty\\
&\leq C\norm{f}_\alpha
\end{align*}
Thus, $\pi_{\ast}$ maps $\mathcal{A}_{\rm sym}^{\alpha}(U^n)$ continuously to $\mathcal{A}^{\frac{\alpha}{\Lambda_n}}(\Sigma^nU)$.  The proof of Theorem~\ref{thm-an} is complete.
\section{Range of the operators $\B_n$ and $\mathcal{E}_n$}
We note here that for $n\geq 2$, the operator $\B_n$  (resp. $\mathcal{E}_n$) is not surjective from the space 
$\mathcal{C}^{k+n-1,\alpha}(\Gamma)$ to $\mathcal{A}_{\rm sym}^{k,\alpha}(U^n)$ (resp. from the space $\mathcal{C}^{(k+1)n-1,\alpha}(\Gamma)$ to the space $\mathcal{A}^{k,\frac{\alpha}{\Lambda_n}}(\Sigma^n U)$.)
Indeed, it is not difficult to characterize the range of the operators $\B_n$ and $\mathcal{E}_n$.  Suppose the domain $U$ 
has $m$ holes, and for $1\leq i \leq m$, we fix a point  $a_i$ in the $i$-th hole. Let $h_p^q$ denote the {\em complete symmetric polynomial} in $q$ variables of degree $p$ given by:
\[ h_p^q(z)=\sum_{\abs{\gamma}=p}z^{\gamma},\] where we set $h_p^q = 0$ if $p<0$. Consider the collection of functions $\mathcal{F}\subset \mathcal{O}_{\rm sym}(U^{n})$, with 
members of the form 
\[ h_r^{n}(z), \text{and } \frac{1}{(z_1-a_i) \dots (z_{n}-a_i)}h_{r-1}^{n}\left(\frac{1}{z_1-a_i},  \dots, \frac{1}{z_{n}-a_i}\right)
\]
 where $r\in\mathbb{N}=\Set{0,1,\dots}$ and $1\leq i\leq m$. We also let $\mathcal{F}_*= \Set{\pi_*(f)| f\in \mathcal{F}}$. Note that the linear span of the collection of functions  $\Set{ z^r, \frac{1}{(z-a_i)^r}| \text{ $r\in\mathbb{N}$ and $1\leq i \leq m$}}$
is dense in $\mathcal{A}^{k+n-1,\alpha}(U)$ (a consequence of the Mergelyan approximation theorem). Recall the representations $\mathcal{B}_n=\divdif^{n-1}\circ \mathcal{T}$ and $\mathcal{E}_n= \pi_*\circ \mathcal{B}_n$
(special case of \eqref{eq-dgammaen} for $\gamma=0$). Using these representations with the facts that $\mathcal{B}_n(z^r)=h_r^n(z)$ 
and 
\[ \mathcal{B}_n\left( \frac{1}{(z-a_i)^r}\right)= \frac{1}{(z_1-a_i) \dots (z_{n}-a_i)}h_{r-1}^{n}\left(\frac{1}{z_1-a_i},  \dots, \frac{1}{z_{n}-a_i}\right),\] we can show the following:
\begin{prop}
The linear span of $\mathcal{F}$ is dense in the range of $\mathcal{B}_n: \mathcal{C}^{k+n-1,\alpha}(\Gamma)\to\mathcal{C}^{k,\alpha}(U^n)$, and that of $\mathcal{F}_*$ is dense in the range of $\mathcal{E}_n: \mathcal{C}^{(k+1)n-1,\alpha}(\Gamma)\to \mathcal{A}^{k,\frac{\alpha}{\Lambda_n}}(\Sigma^n U)$.
\end{prop}
For example the function $z_1z_2\in \mathcal{O}_{\rm sym}(U^2)$ is not in the range of $\mathcal{B}_2$.
\section{Proof of Theorem~\ref{thm-proper}}\label{sec-proper}
We begin by recalling the structure of  proper holomorphic maps between symmetric products of  domains (see \cite[Corollary~3]{cg} and \cite{edi1, ediz1}). Given a holomorphic map $f:U\to V$, there is a unique map $\Sigma^n f: \Sigma^n U\to \Sigma^n V$ such that
for $z\in U^n$ we have
\[ (\Sigma^n f)(\pi(z_1,\dots, z_n))= \pi(f(z_1),\dots, f(z_n)),\]
where $\pi:U^n\to \Sigma^n U$ is the symmetrization map as in  \eqref{eq-pidef}.
The map $\Sigma^n f$ is called the {\em $n$-fold symmetric product} of 
the map $f$. For further properties of the symmetric products of maps, see \cite{cg}. In particular, we need the following integral representation of the 
symmetric power $\Sigma^n f$ of a map $f:U\to \cx$, when $f$ extends continuously to $\partial U$ (see \cite[Proposition~2.4]{cg}):  There is a polynomial 
automorphism $\mathfrak{P}$ of $\cx^n$,  such that we can write
\begin{equation}\label{eq-sym1}
\Sigma^n f = \mathfrak{P}\circ \Psi,
\end{equation}
where $\Psi= (\Psi_1,\dots, \Psi_n):\Sigma^nU\to \cx^n$ is the map whose
$\ell$-th component is given by
\begin{equation}\label{eq-psil}
\Psi_\ell(z) = \frac{1}{2\pi i} \int_{\partial U} \left(f(t)\right)^\ell \cdot \frac{q'_n(z,t)}{q_n(z,t)}dt.
\end{equation} 
Here $q_n(z,t)$ is as in \eqref{eq-p1}:
\[ q_n(z,t) = t^n -z_1t^{n-1}+\dots +(-1)^n z_n = \sum_{j=0}^n (-1)^{j}z_j \, t^{n-j},\]
where we set $z_0=1$ and 
\begin{equation}\label{eq-qprime}
q_n'(z,t)= \frac{\partial q_n}{\partial t}(z,t)= \sum_{j=0}^{n-1}(-1)^j (n-j)z_j t^{n-j-1}.
\end{equation}
If we let $\partial U=\Gamma$, then plugging in  \eqref{eq-qprime} into 
\eqref{eq-psil}  we obtain the representation
\begin{equation}\label{eq-psilsum}
 \Psi_\ell(z)=\sum_{j=0}^{n-1}(-1)^j (n-j) z_j \mathcal{E}_n\left(t^{n-j-1}f|_{\partial U}\right)
\end{equation}

\noindent Using the classical technique of Remmert and Stein (see \cite{remmertstein,nar}), one can show that  if $U$ and $V$ are bounded planar domains, then each proper holomorphic map $F:\Sigma^nU\to \Sigma^n V$ is of the form $F=\Sigma^n f$ for a proper holomorphic map $f:U\to V$ (see \cite{edi1,ediz1, cg}.)
 By hypothesis the boundaries of each of $U$ and $V$ is of class 
$\mathcal{C}^{nk+n+1}$. By the Kellogg-Warschawski theorem, 
(see Result~\ref{result-kw} above), the proper holomorphic map $f: U\to V$ is of class $\mathcal{C}^{nk+n, \theta}(U)$ for each $0<\theta<1$, therefore, a fortiori of class $\mathcal{C}^{n(k+1)-1,\theta}(U)$. Then the restriction $f|_{\partial U}\in \mathcal{C}^{n(k+1)-1,\theta}(\partial U)$. Applying Theorem~\ref{thm-an}, each term in the sum \eqref{eq-psilsum} belongs to $\mathcal{A}^{k, \frac{\theta}{\Lambda_n}}(\Sigma^nU)$. Therefore the map $\Psi\in \mathcal{A}^{k, \frac{\theta}{\Lambda_n}}(\Sigma^nU)$, and since $F= \Sigma^n f = \mathfrak{P}\circ \Psi$, where $\mathfrak{P}$ is a polynomial automorphism of $\cx^n$, we have that $F\in \mathcal{A}^{k, \frac{\theta}{\Lambda_n}}(\Sigma^nU)$. The proof is complete.

\end{document}